\documentclass[a4paper,reqno,oneside]{amsart}

\usepackage{amsmath,amsfonts,amssymb,amsthm}
\usepackage{graphicx}
\usepackage{color}
\usepackage[backref,colorlinks]{hyperref}
\usepackage{enumitem}
\usepackage{verbatim}

\setitemize{leftmargin=\parindent}
\setenumerate{leftmargin=*}
 
\def\itm#1{\rm ({#1})} 
\def\itmit#1{\itm{\it #1\,}} 
 
\def\abc{\itmit{\alph{*}}}

\newcommand{\cF}{\mathcal{F}}

\newcommand{\link}{\mathrm{link}}

\newtheorem{theorem}{Theorem}
\newtheorem{lemma}[theorem]{Lemma}
\newtheorem{claim}[theorem]{Claim}

\newtheorem{fact}[theorem]{Fact}

\theoremstyle{definition}

\theoremstyle{remark}

\renewcommand{\subset}{\subseteq}

\renewcommand{\rho}{\varrho}
\renewcommand{\phi}{\varphi}

\newcommand{\PP}{\mathbb{P}}

\newcommand{\iti}[1]{\ensuremath{\mathrm{ (\textit{#1\,})}}}

\setitemize{leftmargin=\parindent}
\setenumerate{leftmargin=*}

\newcommand{\cH}{\mathcal{H}}
\newcommand{\Exp}{\mathbb{E}}

\newcommand{\cM}{\mathcal{M}}
\newcommand{\Krt}{K^{(r)}_t}
\newcommand{\Ktt}{K^{(3)}_t}

\newcommand{\Km}[1]{K_{#1}^{(3)}}
\newcommand{\red}{\mathrm{red}}
\newcommand{\blue}{\mathrm{blue}}
\newcommand{\girth}{\text{girth}}
\newcommand{\cP}{\mathcal{P}}
\newcommand{\dist}{\mathrm{dist}}
\begin{document}

\title[Minimal Ramsey Hypergraphs]{Minimum degrees and codegrees  of minimal Ramsey $3$-uniform hypergraphs}
\author[Dennis Clemens]{Dennis Clemens}
\address{Technische Universit\"at Hamburg-Harburg, Institut f\"ur Mathematik, 
  Schwarzenberg-str. 95, 21073 Hamburg, Germany}
\email{dennis.clemens@tuhh.de}
\author[Yury Person]{Yury Person}
\address{Goethe-Universit\"at, Institut f\"ur Mathematik,
  Robert-Mayer-Str. 10, 60325 Frankfurt am Main, Germany}
\email{person@math.uni-frankfurt.de}

 \thanks{
    YP is partially supported by DFG grant PE 2299/1-1.
  }

\date{\today}

\begin{abstract}
 A uniform hypergraph $H$ is called $k$-Ramsey for a hypergraph $F$, if no 
matter how one colors the edges of $H$ with $k$ colors, there is always a monochromatic 
copy of $F$. We say that $H$ is minimal $k$-Ramsey for $F$, if $H$ is $k$-Ramsey for $F$ but 
every proper subhypergraph of $H$ is not. Burr, Erd\H{o}s and Lovasz [S.\ A.\ Burr, P.\ 
Erd\H{o}s, and L.\ Lov\'asz, \emph{On graphs of Ramsey type}, Ars Combinatoria 1 (1976),
no. 1, 167--190] studied various parameters of minimal Ramsey graphs. 
In this paper we initiate the  study of 
 minimum degrees and codegrees of minimal Ramsey $3$-uniform hypergraphs. We show that 
the smallest minimum vertex degree over all minimal $k$-Ramsey $3$-uniform hypergraphs for 
$\Ktt$ is exponential in some polynomial in $k$ and $t$. We also study the smallest possible minimum 
codegrees over minimal $2$-Ramsey $3$-uniform hypergraphs.
\end{abstract}

\maketitle

\section{Introduction and New Results}
A graph $G$ is said to be Ramsey for a graph $F$ if no matter how one colors the edges of $G$ with two colors, 
say red and blue, there is a monochromatic copy of $F$ (we write $G\longrightarrow (F)_2$ for this). 
 A classical result of Ramsey~\cite{Ram30} states that 
for every  $F$ there is an integer $n$ such that $K_n$ is Ramsey for $F$. Moreover, generalizations to more than two colors and to hypergraphs hold as well~\cite{Ram30}. 
We say that $G$ is minimal Ramsey for $F$ 
if $G$ is Ramsey for $F$ but every proper subgraph of $G$ is not.  More generally, we denote by $\cM_k(F)$ the set of minimal graphs $G$ with the property that 
no matter how one colors the edges of $G$ with $k$ colors, there is a monochromatic copy of  $F$ in it, and refer to these as minimal $k$-Ramsey graphs for $F$. 
There are many challenging open questions concerning 
 the study of various parameters of minimal $k$-Ramsey graphs for various $F$. The most studied ones are the classical (vertex) Ramsey numbers $r_k(F):=\min_{G\in\cM_k(F)}v(G)$ and 
 the size Ramsey number $\hat{r}_k(F):=\min_{G\in\cM_k(F)}e(G)$, where $v(G)$ is the number of vertices in $G$ and $e(G)$ is its number of edges. 
To determine the classical Ramsey number $r_2(K_t)$ is a notorously difficult problem and essentially
 the best known bounds are $2^{(1+o(1))t/2}$ and $2^{(2+o(1))t}$ due to Spencer~\cite{Spe75} and Conlon~\cite{Con09}.

Burr, Erd\H{o}s and Lov\'asz~\cite{BEL76} were the first to study  other
 possible parameters of the class $\cM_2(K_t)$. In particular they 
determined the minimum degree $s_2(K_t):=\min_{G\in\cM_2(K_t)}\delta(G)=(t-1)^2$ which looks surprising given the exponential bound on the minimum degree of $K_n$ with 
$K_n\longrightarrow (K_t)_2$ and $n=r_2(K_t)$ (it is not difficult to see that such $K_n$ is indeed minimal $2$-Ramsey for $K_t$). 
 Generalizing their results, Fox, Grinshpun, 
Liebenau, Person and Szab\'o~\cite{FGLPS14} studied the minimum degree $s_k(K_t):=\min_{G\in\cM_k(K_t)}\delta(G)$ for more colors showing a general bound on 
$s_k(K_t)\le 8 (t-1)^6 k^3$ and proving quasiquadratic bounds in $k$ on $s_k(K_t)$ for fixed $t$. Further results concerning minimal Ramsey graphs were studied in~\cite{bnr1984,fl2006,rs2008,szz2010,FGLPS14a}.

In this paper we initiate the study of minimal Ramsey 
 $3$-uniform hypergraphs and provide first bounds on various notions of minimum degrees 
for minimal Ramsey hypergraphs.
 Generally, an $r$-uniform hypergraph 
$H$ is a tuple $(V,E)$ with vertex set $V$ and $E\subseteq \binom{V}{r}$ being its edge set. We define  $\link(v)$, the \emph{link} of a vertex 
$v\in V$, to be the edges of $H$ that contain $v$, minus the vertex $v$ (thus, these form an $(r-1)$-uniform hypergraph). 
Formally, the edge set of $\link(v)$ is $\{e\setminus \{v\}:\ v\in e\in E \}$. The random $r$-uniform hypergraph 
$H^{(r)}(n,p)$ is the probability space of all labeled $r$-uniform hypergraphs on 
the vertex set $[n]$ where each edge exists with probability $p$ independently  of the other edges.
 In this paper we will be dealing exclusively with $3$-uniform hypergraphs, thus 
the links of their vertices are just the edges of some graph.

Ramsey's theorem holds for $r$-uniform hypergraphs as well as shown originally by Ramsey himself~\cite{Ram30}, and we 
write $G\longrightarrow (F)_k$, if no matter how one colors the edges of the $r$-uniform hypergraph $G$, there is a monochromatic 
copy of $F$. We denote by $\Krt$ the complete $r$-uniform hypergraph 
 with $t$ vertices, i.e.\  $\Ktt=([t],\binom{[t]}{r})$, and by the hypergraph Ramsey number $r_k(F)$ the smallest $n$ such 
that $K^{(r)}_n\longrightarrow (F)_k$. While in the graph case the known 
bounds on $r_2(K_t)$ are only polynomially far apart, already in the case of $3$-uniform hypergraphs the bounds on $r_2(\Krt)$ differ in one exponent: 
$2^{c_1 t^2}\le r_2(\Ktt)\le 2^{2^{c_2 t}}$ for some absolute positive constants $c_1$ and $c_2$. More generally, it holds 
$t_{r-1}(c_1 t^2)\le r_2(\Krt)\le t_{r}(c_2 t)$ for some absolute constants $c_1=c_1({r}),c_2=c_2({r})>0$ and where $t_i(x)$ is the tower function defined by $t_1(x):=x$, $t_{i}(x):=2^{t_{i-1}(x)}$. 
For further information on hypergraph Ramsey numbers we refer the reader to the standard book on Ramsey theory~\cite{GRS90} and for newer results to the work of 
Conlon, Fox and Sudakov~\cite{CFS2015}.

  Given $\ell\in[r-1]$, we define the degree $\deg(S)$ of an $\ell$-set $S$ in an 
$r$-uniform hypergraph $H=(V,E)$ as the number of edges that contain $S$ and the minimum $\ell$-degree 
$\delta_{\ell}(H):=\min_{S\in\binom{V}{\ell}}\deg(S)$. For two vertices $u$ and $v$ we simply write 
$\deg(u,v)$ for the \emph{codegree} $\deg(\{u,v\})$. 

Similar to the graph case we extend verbatim  the notion of minimal Ramsey graphs to minimal Ramsey $r$-uniform hypergraphs  $\cM_k(F)$  in a natural way. That is,
  $\cM_k(F)$ is the set of all minimal $k$-Ramsey $r$-uniform hypergraphs $H$ with $H\longrightarrow (F)_k$.
We define 
\begin{equation}\label{eq:sklK}
s_{k,\ell}(\Krt):=\min_{G\in\cM_k(\Krt)}\delta_{\ell}(G), 
\end{equation}
which extends the introduced graph parameter $s_2(K_t)$. It will be shown actually that 
$s_{2,2}(\Km{t})$ is zero and thus it makes sense to ask for the second smallest value of the codegrees.
 This motivates the following parameter $s'_{k,\ell}(\Krt)$:
\[
s'_{k,\ell}(\Krt):=\min_{G\in\cM_k(\Krt)}\left(\min \left\{\deg_G(S)\colon S\in\binom{V(G)}{\ell}, \deg_G(S)>0 \right\}\right).  
\]
 
  We prove the following results on 
the minimum degree and codegree of minimal Ramsey $3$-uniform  hypergraphs for cliques $\Km{t}$.

\begin{theorem}\label{thm:mindegree} 
The following holds for all $t\ge 4$ and $k\ge2$ 
\begin{equation}\label{eq:mindeg}
2^{\frac{1}{2}kt(1-o(1))}\le \binom{r_k(K_{t-1})}{2}  \leq s_{k,1}(\Km{t}) \leq k^{20kt^4}. 
\end{equation}
\end{theorem}

For the lower bound see \cite{CFS2015}.

\begin{theorem}\label{thm:main_result_codeg}
Let $t\ge 4$ be an integer. Then, 
\[
 s_{2,2}(\Km{t})=0\text{  and  }s'_{2,2}(K_t^{(3)})=(t-2)^2.
\]
\end{theorem}
Observe that with
$s'_{2,2}$ we ask for the smallest {\em positive} codegree, while for $s_{2,2}$
we also allow the codegree to be zero.
This in particular means that in 
\emph{any} minimal $2$-Ramsey hypergraph $H$ for $\Km{t}$ we have that a pair of vertices $u$ and $v$ 
are either not contained in a common edge or have codegree at least $(t-2)^2$. 

\subsection*{Methods} The methods we are going to use are  generalizations of signal senders introduced 
first by Burr, Erd\H{o}s and Lov\'asz in~\cite{BEL76}, and generalized later by 
Burr, Ne{\v{s}}et{\v{r}}il and R{\"o}dl~\cite{bnr1984} and by R\"odl and Siggers~\cite{rs2008}, that we 
 combine with probabilistic arguments  analyzing certain 
properties of random $3$-uniform hypergraphs.

\subsection*{Organization of the paper}
In the next section, Section~\ref{sec:bel}, we generalize 
``almost'' Ramsey graphs, i.e.\ graphs whose edge colorings without a monochromatic copy of some complete graph $K_t$
 impose certain color pattern, first introduced by Burr, Erd\H{o}s and Lov\'asz~\cite{BEL76} to hypergraphs. 
 Then we study in Section~\ref{sec:deg_graphs} the vertex degree for minimal $k$-Ramsey $3$-uniform hypergraphs for $\Ktt$,
 while in Section~\ref{sec:codeg_graphs}
 we look into the case of codegrees in  minimal $2$-Ramsey $3$-uniform hypergraphs for $\Ktt$.

\section{\texorpdfstring{BEL-Gadgets for $3$-uniform hypergraphs}{BEL-Gadgets for 3-uniform hypergraphs}}\label{sec:bel}
First we show a lemma that asserts the existence of a $3$-uniform hypergraph $H$ and two edges $f$, $e$ $\in E(H)$ 
with $|f\cap e|=2$ and $e(H[e\cup f])=2$ so that $H$ is not 
$k$-Ramsey for $\Ktt$ with the property that any $k$-coloring of $E(H)$ 
without a  monochromatic  
$\Ktt$ colors the edges $e$ and $f$ differently. 
We will refer to such hypergraphs that impose certain structure on $\Ktt$-free colorings as \emph{BEL-gadgets}. 
Moreover, we refer in the following to a coloring without a 
monochromatic copy of $F$ as an \emph{$F$-free coloring}.

\begin{lemma}\label{lem:senders}
Let $t\ge 4$ and $k\ge 2$ be integers. Then there exist a $3$-uniform hypergraph 
$\cH$ and two edges $e_{\cH},f_{\cH}\in E(\cH)$ with $|f_{\cH}\cap e_{\cH}|=2$ and $e(\cH[e_{\cH}\cup f_{\cH}])=2$ 
such that the following properties hold:
\begin{enumerate}
\item $\cH\not\rightarrow \left(\Ktt\right)_k$,
\item for every $k$-coloring $c$ of $E(\cH)$ which avoids 
 monochromatic copies of $\Ktt$ we have that $c(e_{\cH})\neq c(f_{\cH})$.\label{prop:colors}
\end{enumerate}
\end{lemma}

\begin{proof}
Set $m=r_k(\Ktt)$ and define a hypergraph $F'$ on the vertex set $[m]$ as follows: delete from $\Km{m}$ 
all edges that contain vertices $m-1$ and $m$. It is easy to see that then $F'\not\longrightarrow (\Ktt)_k$. 
Indeed, fix a $k$-coloring of $E(\Km{m-1})$ without a monochromatic $\Ktt$, 
then extend this coloring to $E(F')$ by coloring each edge $(x,y,m)$ with the color of $(x,y,m-1)$.
Since every copy of $\Ktt$ in $F'$ may contain at most one of the vertices $m-1$ and $m$, we see 
$F'\not\longrightarrow (\Ktt)_k$.

Define $F_i:=\left([m], E(F')\cup\left\{\{j,m-1,m\}\colon j\le i\right\}\right)$ and 
set $F:=F_\ell$ where $\ell$ is maximal such that $F_\ell$ is not $k$-Ramsey for $\Ktt$ but 
$F_{\ell+1}$ is (this is possible since $F_{m-2}=\Km{m}$ is $k$-Ramsey for $\Ktt$ by the choice of $m=r_k(\Ktt)$). 

For a coloring $\psi\colon E(F)\to [k]$ without a monochromatic copy of $\Ktt$ 
we define an \emph{admissible pattern} $(a_1,\ldots,a_k)$, where $a_i$ denotes the number
of edges in the color $i$ containing both vertices $m-1$ and $m$.
 Moreover, with $\cP$ we denote the set of all admissible patterns. In particular, by the choice of $\ell$ we have that $\cP\neq\emptyset$. 

Notice that $\sum_{i\in[k]}a_i=\ell$ for every $(a_1,\ldots,a_k)\in{\cP}$, and
$a_c\not\in \{0,\ell\}$ for every $c\in[k]$. Indeed if, say, there is a pattern $(a_1,\ldots,a_k)\in{\cP}$ with $a_j=0$ for some $j\in[k]$, 
then we could take a corresponding $k$-coloring of the edges of $F_{\ell}$ 
avoiding monochromatic copies of $\Ktt$
with pattern $(a_1,\ldots,a_k)$, which then we would extend  to a $k$-coloring of $E(F_{\ell+1})$
without a monochromatic copy of $\Ktt$ just 
by coloring the edge $\{\ell+1,m-1,m\}$ in color $j$. Indeed, this new edge cannot parti\-cipate in
a monochromatic copy of $\Ktt$ in this coloring, as its color is $j$, while all
other edges containing both $m-1$ and $m$ have colors different from $j$. But this is a contradiction
to the definition of $\ell$.

Moreover, notice that the following holds: If $\varphi\colon [\ell]\to [k]$
is a coloring of the first $\ell$ vertices of $F$ such that $(|\varphi^{-1}(1)|,\ldots,|\varphi^{-1}(k)|)\in \cP$,
then there exists a coloring $c\colon E(F)\to [k]$ avoiding monochromatic copies
of $K_t^{(3)}$ such that $c(i,m-1,m)=\varphi(i)$ for every $i\in [\ell]$.

Now, let $H$ be an $\ell$-uniform hypergraph. We say that a coloring $\psi\colon V({H})\to [k]$
is admissible, if for every edge $e\in E({H})$ we have $(c_1,\ldots,c_k)\in {\cP}$ where $c_i$ denotes the number
of vertices in $e$  colored $i$. 

Now we proceed analogously to Claim~2 from \cite{BEL76}.  
We  find an $\ell$-uniform hypergraph $H^*$ with $\girth(H^*)\ge 3$ 
(this means that any two distinct edges $e$ and $f$ satisfy $|e\cap f|\le 1$) and 
two vertices $x,y\in V({H}^*)$ with $\deg_{H^*}(x,y)=0$ such that there exist 
admissible colorings for ${H}^*$ and in every such coloring the color of $x$ differs from the color of $y$.
 For completeness we provide this elegant argument here. 
We start with  an $\ell$-uniform hypergraph $H$ with $\girth({H})\ge 3$ 
and chromatic number $\chi({H})\ge k+1$. 
It was shown that such hypergraphs exist by Erd\H{o}s and Hajnal in~\cite{EH66}.

Then, as every $k$-coloring of the vertices of 
$H$ yields a monochromatic edge, while $(\ell,0,\ldots,0)$,\ldots ,$(0,\ldots,0,\ell)\notin\cP$, $H$ does not have admissible colorings. 
Now, we can take a subhypergraph $H'$
of $H$ which is minimal (with respect to the number of edges) for the property of not having admissible $k$-colorings. 
For an arbitrary edge $f=\{x_1,\ldots,x_\ell\}\in H'$ and arbitrary vertices $y_1,\ldots,y_\ell\not\in V(H')$,
we define a sequence of hypergraphs $H_i$ on $V(H')\cup \{y_1,\ldots, y_i\}$
with $H_i=H'-f+f_i$, where $f_i=\{y_1,\ldots,y_i,x_{i+1},\ldots,x_\ell\}$. 
By the definition,  $H_0=H'$ does not have admissible colorings while $H_\ell$ does, so there is 
 a minimal index $i\in[\ell]$ such that $H_{i-1}$ does not have admissible colorings,
but $H_i$ does. We now set $H^*=H_i$ and $x:=x_i$, $y:=y_i$. It is clear that 
$\girth(H^*)\ge 3$, $\deg_{H^*}(x,y)=0$ and that $H^*$ has admissible colorings. Moreover, 
for any such admissible $k$-coloring $x$ and $y$ need to have distinct colors as otherwise, by taking 
an admissible coloring of $H_i$ with $x$ and $y$ colored the same and then identifying $x$ with $y$ would yield an admissible coloring of  
$H_{i-1}$, a contradiction.

Finally, we define a $3$-uniform hypergraph $\cH$ as follows. First we introduce for each $e\in E(H^*)$  a set 
$V_e:=e\cup \{m-1,m\}\cup (\{e\}\times\{\ell+1,\ldots, m-2\})$ and then we define a $3$-uniform hypergraph  
$F_e$ which is a copy of 
 $F=F_{\ell}$ that contains all vertices from $e$ as follows: 
\[
F_e:=\left(V_e,\binom{V_e}{3}\setminus\left\{\{(e,i),m-1,m\}\colon i=\ell+1,\ldots, m-2\right\}  \right). 
\]
The hypergraph $\cH$ is then the union over all $F_e$'s: $\cH:=\cup_{e\in E(H^*)}F_e$.  
In other words, we obtain $\cH$ by placing $F_e$, a copy of $F$, for each edge $e\in E(H^*)$ 
so that the vertices $\{1,\ldots,\ell\}$ of $F$ are identified with $e$.  
Further, we set $e_{\cH}=\{m-1,m,x\}$ and $f_{\cH}=\{m-1,m,y\}$. 
Before showing that 
${\cH}$, $e_{\cH}$ and $f_{\cH}$ fulfill the requirements (1) and (2), we establish the following claim.
\begin{claim}\label{cl:Kts}
 Any copy $K$ of  $\Km{t}$  in $\cH$  is contained in $F_e$ for some $e\in E(H^*)$.
\end{claim}
\begin{proof}
Assume first $V(K)\setminus (\{m-1,m\}\cup V(H^*))\neq\emptyset$ holds. 
 Thus $K$ contains a vertex of the form $(e,s)$, whose link is a graph on  
$m-1$ vertices which must form the set  $V_e\setminus\{(e,s)\}$, by construction of $\cH$. 
This, with $\cH[V_e]=F_e$, then implies that $K\subseteq F_e$.

From now on we may assume that  $V(K)\subseteq V(H^*)\cup\{m-1,m\}$. 
First we assume that $K\cong \Km{4}$ and $m-1, m\in V(K)$. Thus, 
the remaining two vertices, call them $a$ and $b$, must lie in 
some edge $e\in E(H^*)$ (since $\{m,a,b\}$ is an edge in 
$\cH\left[V(H^*)\cup\{m-1,m\}\right]$), which  implies $K\subseteq F_e$.
 Finally, we may assume that $|V(K)\cap V(H^*)|\ge 3$ and setting $S:=V(K)\cap V(H^*)$ 
 we have  $K[S]\cong \Km{s}$, $s\geq 3$. Since $\cH\left[V(H^*)\right]$ consists of cliques $\Km{\ell}$ that intersect in at most 
one vertex as $\girth(H^*)\ge 3$, this implies that $S$ has to be contained in some $e\in E(H^*)$. Again this 
yields $K\subseteq F_e$.
\end{proof}

Recall that we defined  $e_{\cH}=\{m-1,m,x\}$ and $f_{\cH}=\{m-1,m,y\}$. 
By construction of $\cH$ and since $\deg_{H^*}(x,y)=0$, it is 
clear that $\{x,y,m-1\}$ and  $\{x,y,m\}$ are nonedges in $\cH$.                                                                                                            
 We now prove that this choice of $\cH$, $e_{\cH}$ and $f_{\cH}$ fulfills the requirements (1) and (2) of our lemma:
\begin{enumerate}
\item By construction there exists an admissible coloring 
$c\colon V(H^*)\to [k]$. Notice that two hypergraphs $F_e$ and $F_f$ for distinct $e, f \in E(H^*)$ 
have in common both vertices $m-1$ and $m$
 and additionally at most one further vertex $v$ (and if so also the edge $\{v,m-1,m\}$), 
 by construction and since $\girth(H^*)\ge 3$.
 Since $\cH$ consists of copies of $F$ that intersect pairwise in at most one edge 
(containing both vertices $m-1$ and $m$), we can find colorings of these copies without monochromatic $\Ktt$ so that these 
colorings agree on common edges $\{v,m-1,m\}$. Indeed, for every edge $e\in E(H^*)$ we have an admissible color pattern 
$(d_1,\ldots,d_k)\in\cP$ which depends on $c$. Thus, there exists a coloring $\varphi_e\colon E(F_e)\to [k]$ 
without monochromatic $\Km{t}$ so that 
$\varphi_e(\{v,m-1,m\})=c(v)$ for all $v\in e$. 

We need to show that the union of $\varphi_e$ 
over all $e\in E(H^*)$ gives us a $k$-coloring $\varphi$ of $E(\cH)$
without monochromatic copies of $\Km{t}$. By Claim~\ref{cl:Kts}, any copy of $\Km{t}$ is contained in 
 $F_e$ for some $e\in E(H^*)$. Since $E(F_e)$ does not contain any monochromatic $\Km{t}$ under $\varphi_e$, the requirement (1) is verified. 
\item Now, let $c\colon E(\cH)\to [k]$ be a coloring on the edge set of $\cH$ which avoids monochromatic copies
of $\Km{t}$. Define $\varphi: V(H^*)\to [k]$ with $\varphi(v):=c(\{v,m-1,m\})$. 
Then $\varphi$ is an admissible coloring of $H^*$ and thus, by the properties of $H^*$ we know
that $c(e_{\cH})=\varphi(x)\neq \varphi(y)=c(f_{\cH})$.\qedhere
\end{enumerate}
\end{proof} 

We introduce the following definition of a path in hypergraphs. In an $r$-uniform 
path (or $r$-path for short notation) with $t$ edges $e_1$,\ldots,$e_t$ the vertices of $\cup_{i\in [t]} e_i$ 
are ordered linearly and the edges are \emph{consecutive} segments with the property 
that $e_i\cap e_{i+1}\neq \emptyset$ for all $i\in[t-1]$. We will refer to the edges $e_1$ 
and $e_t$  as \emph{ends} of such a path. 
In particular, in our notation the path is a 
vertex-connected subhypergraph of a so-called tight path on the vertex set 
$\cup_{i\in [t]} e_i$ (where in a tight path it is $|e_i\cap e_{i+1}|=r-1$).

Further we say that two edges $e$ and $f$ have distance $\dist_H(e,f):=s$ in $H$ if any 
$r$-uniform path in $H$ with ends $e$ and $f$ contains at least $s$ vertices and there exists
at least one such path with exactly $s$ vertices. 
We call a path from $e$ to $f$ with $\dist_H(e,f)$ vertices a shortest path. If no such path exists, 
we set $\dist_H(e,f):=\infty$.

First we show a lemma that allows us to obtain a ``rainbow star''. 
\begin{lemma}\label{lem:rainbow}
Let $t\ge 4$ and $k\ge 2$ be integers. Then there exist a $3$-uniform hypergraph 
$\cH$, a $2$-element set $S\subset V(\cH)$ and edges 
$e_{1}$, \ldots, $e_k\in E(\cH)$ with $e_{i}\cap e_{j}=S$ (for all $i\neq j\in[k]$), $|\cup_{i\in[k]} e_{i}|=k+2$ and $e(\cH[\cup_{i\in[k]} e_{i}])=k$ 
such that the following properties hold:
\begin{enumerate}
\item $\cH\not\rightarrow \left(\Ktt\right)_k$,
\item for every $k$-coloring $c$ of $E(\cH)$ which avoids 
 monochromatic copies of $\Ktt$ we have that $\{c(e_i)\colon i\in[k]\}=[k]$, that is the colors of $e_i$s are all distinct. \label{prop:rainbow}
\end{enumerate}
\end{lemma}
\begin{proof}
Take $\binom{k}{2}$ vertex-disjoint copies $(\cH_{ij})_{1\le i<j\le k}$ of the hypergraph $\cH'$ as 
guaranteed to us by Lemma~\ref{lem:senders}, and let $e_{ij}$ and $f_{ij}$ be the corresponding edges of $\cH'$ that satisfy Property~\eqref{prop:colors} of Lemma~\ref{lem:senders}.  We start with the hypergraph $H$ on the vertex set $[k+2]$ and with edge set $\left\{\{i,k+1,k+2\}\colon i\in[k]\right\}$, and we set $S:=\{k+1,k+2\}$. 

We construct the hypergraph $\cH$ as follows. For each $i<j\in[k]$ we identify the vertices $k+1$ and $k+2$ (arbitrarily) with the two vertices from $C_{ij}:=e_{ij}\cap f_{ij}$ and the  only vertex from 
$e_{ij}\setminus C_{ij}$ is identified with $i$ while the only vertex from 
$f_{ij}\setminus C_{ij}$ is identified with $j$. Otherwise the hypergraphs $\cH_{ij}$ don't intersect each other in further vertices. We claim that the properties from Lemma~\ref{lem:rainbow} are satisfied. Indeed, since  $\cH_{ij}\not\rightarrow \left(\Ktt\right)_k$ and by the symmetry of the colors, we can assume that there is a $\Ktt$-free coloring
 $\varphi_{ij}$ of $\cH_{ij}$ such that $\varphi(e_{ij})=i$ and $\varphi(f_{ij})=j$ (and $i<j$). We obtain the coloring $\varphi$ of $\cH$ by coloring the corresponding edges according to appropriate $\varphi_{ij}$s.  This is possible since the edge $\{i,k+1,k+2\}$ is identified  with $e_{ij}$ and $f_{\ell i}$ for $\ell<i<j$, and these are colored with the color $i$. The coloring $\varphi$ is $\Ktt$-free, since each copy of $\Ktt$ is contained in one of the $\cH_{ij}$s. 
 To  see Property~\eqref{prop:rainbow}, we use the Property~\eqref{prop:colors} of Lemma~\ref{lem:senders}, which asserts that in any  $\Ktt$-free coloring of $\cH$ the edges $\{i,k+1,k+2\}$ and $\{j,k+1,k+2\}$ are colored differently (with $i<j$).
\end{proof}

The next lemma allows us to construct a BEL-gadget that colors two edges the same.
\begin{lemma}\label{lem:monocolor}
Let $t\ge 4$ and $k\ge 2$ be integers. Then there exist a $3$-uniform hypergraph 
$\cH$ and edges $e$ and $f$ with $|e\cap f|=2$ and $e(\cH[e\cup f])=2$ 
such that the following properties hold:
\begin{enumerate}
\item $\cH\not\rightarrow \left(\Ktt\right)_k$,
\item for every $k$-coloring $c$ of $E(\cH)$ which avoids 
 monochromatic copies of $\Ktt$ we have that $c(e)=c(f)$.\label{prop:equal}
\end{enumerate}
\end{lemma}
\begin{proof}
 We take two vertex-disjoint copies of $\cH_1$ and $\cH_2$ as asserted by Lemma~\ref{lem:rainbow}, along with 
 the corresponding edges $e_{1,1}$,\ldots, $e_{1,k}$ for $\cH_1$ and 
 $e_{2,1}$,\ldots, $e_{2,k}$ for $\cH_2$ respectively. Recall that there exist $S_1$ and $S_2$ such that 
 $e_{\ell,i}\cap e_{\ell,j}=S_\ell$ for all $i<j\in [k]$ and $\ell\in[2]$. We obtain the hypergraph $\cH$ by identifying 
 the edge $e_{1,i}$ with $e_{2,i}$ for all $2\le i \le k$ such  that 
 the vertices from $S_1$ are identified  with those from $S_2$. 
 
 We set $e:=e_{1,1}$ and $f:=e_{2,1}$ and claim that $\cH$ fulfills the requirements. By the symmetry of the colors, we may assume that $e_{\ell,i}$ may be colored with the color $i$ for all $i\in[k]$ and $\ell\in[2]$, and then we may extend the coloring by coloring the (otherwise disjoint) copies $\cH_1$ and $\cH_2$ separately. Since any copy of $\Ktt$ is contained fully either in $\cH_1$ or in $\cH_2$, we see $\cH\not\rightarrow \left(\Ktt\right)_k$. On the other hand, any $\Ktt$-free coloring $\varphi$ of $\cH$ is a $\Ktt$-free coloring of $\cH_1$ and $\cH_2$, and from the properties from Lemma~\ref{lem:rainbow} we have that the edges $e_{\ell,1}$,\ldots, $e_{\ell,k}$ are colored differently for each $\ell\in[2]$ and, by the construction, $\varphi(e_{1,i})=\varphi(e_{2,i})$ for all $2\le i\le k$. Thus, 
 we also have $\varphi(e_{1,1})=\varphi(e_{2,1})$.
\end{proof}
 
 Finally, we construct BEL-gadgets with monochromatic edges in every $\Ktt$-free coloring that are ``far'' from each other. 
\begin{lemma}\label{lem:BEL_far_edges}
 Let $s$, $t\ge 4$ and $k\ge 2$ be integers. There exist a $3$-uniform 
hypergraph $H$ and two edges $e,f\in E(H)$ such that the following properties hold:
\begin{enumerate}
\item $H\not\rightarrow \left(\Ktt\right)_k$,
\item $e$ and $f$ have distance at least $s$, and
\item for every $k$-coloring $\varphi$ on $E(H)$ which avoids 
 monochromatic copies of $\Ktt$ we have that $\varphi(e)= \varphi(f)$.
\end{enumerate}
\end{lemma}

\begin{proof}
 First we construct a hypergraph $\cH$ which is not $k$-Ramsey for $\Km{t}$, but 
contains two edges $e$ and $f$ at distance  $5$ that are colored the same 
by any $k$-coloring of $E(\cH)$ without monochromatic $\Km{t}$. We apply Lemma~\ref{lem:monocolor} twice  
and obtain  $3$-uniform hypergraphs  
 $\cH_1$ with edges $e_{\cH_1}=\{a,b,x_1\}$ and $f_{\cH_1}=\{a,b,y_1\}$ and $\cH_2$ 
with edges $e_{\cH_2}=\{c,d,x_2\}$ and $f_{\cH_2}=\{c,d,y_2\}$ respectively. 
Furthermore, we may assume $V(\cH_1)\cap V(\cH_2)=\emptyset$.
 We define a new hypergraph $\cH$ by taking 
both $\cH_1$ and $\cH_2$ and identifying $y_1$ with $d$, $b$ with $c$, and $a$ with $y_2$. 
Observe that in $\cH$ any copy of $\Km{t}$ is completely contained within one of the $\cH_i$'s. This 
implies that $\cH\not\rightarrow (\Km{t})_k$. 
Indeed, according to Lemma~\ref{lem:senders} we can color  $\cH_1$ and $\cH_2$
without monochromatic $\Km{t}$. Moreover, by swapping the colors appropriately if necessary,
we may do so that the edges $f_{\cH_1}\in E(\cH_1)$ and $f_{\cH_2}\in E(\cH_2)$ receive the same color.
This gives us a $\Km{t}$-free coloring of $E(\cH)$.

Next we use the Property~(2) of Lemma~\ref{lem:monocolor} which asserts that any 
$\Km{t}$-free coloring colors the edges $\{a,b,x_1\}$ and $\{a,b,y_1\}$  the same, and 
the colors of $\{c,d,x_2\}$ and $\{c,d,y_2\}$ are the same as well. Since $\{a,b,y_1\}=\{c,d,y_2\}$ in $\cH$,  
 the edges $f:=\{c,d,x_2\}$ and $e:=\{a,b,x_1\}$ are colored the same through any
  $\Km{t}$-free coloring of $\cH$. 
We thus arrived at a hypergraph $\cH$ that satisfies the following properties:
\begin{enumerate}[label=\abc]
\item there are two edges $e$ and $f$ at distance $5$,
\item $\cH\not\rightarrow \left(\Ktt\right)_k$,
\item for every $k$-coloring $c$ on $E(\cH)$ which avoids monochromatic copies of $\Ktt$ we have that $c(e)= c(f)$.
\end{enumerate}
Next we proceed iteratively. We take two isomorphic hypergraphs $H_1$  and $H_2$,
along with edges $e_1$, $f_1$ and $e_2$, $f_2$ respectively, which satisfy~\iti{b} and~\iti{c}. 
Assuming that 
$\dist_{H_1}(e_1,f_1)=d=\dist_{H_2}(e_2,f_2)$ for some $d\geq 5$, we now aim to construct
a hypergraph $H'$, along with edges $e,f$, such that ~\iti{b} and~\iti{c} hold
and $\dist_{H'}(e,f)\geq d+1$. For the construction, we identify  
the edge $f_1$ with $e_2$ such that none of the vertices of $e_1$ and $f_2$ 
are identified, and we set $e=e_1$ and $f=f_2$. This way the properties~\iti{b}  and~\iti{c} 
are naturally preserved in $H'$.

Thus, it remains to show that the distance between $e_1$ and $f_2$ is 
 at least $d+1$ in $H'$. 
 Let $v_1$, \ldots, $v_\ell$ 
be the  vertices 
 of a shortest path from $e_1$ to $f_2$ in $H'$ in the linear order, i.e.\ $\{v_1,v_2,v_3\}=e_1$ and 
$\{v_{\ell-2},v_{\ell-1},v_\ell\}=f_2$. 
Let $i\ge 4$ be the smallest index 
such that $v_i\not\in V(H_1)$. If $i<d-1$, then we have $v_{i-1}\in f_1$
 and in case $\{v_{i-3},v_{i-2},v_{i-1}\}\not\in E(H_1)$ holds then we additionally have $\{v_{i-4},v_{i-3},v_{i-2}\}\in E(H_1)$ and $v_{i-2}\in f_1$. In any case we would obtain a $3$-path from $e_1$ to $f_1$ with at most $d-1$ vertices which consists of some 
   edges of $P$ contained in $\{v_1, \ldots, v_{i-1}\}$ and of the edge  $f_1$, a contradiction to $\dist_{H_1}(e_1,f_1)=d$. 
 Thus we may assume $i\ge d-1$. If, additionally, $d>5$ then it follows, that 
none of the vertices from $f_2$ are among $\{v_1,\ldots,v_{i-1}\}$ resulting in $\dist_{H'}(e_1,f_2)\ge d+1$.
If $d=5$, then since none of the vertices of $e_1$ and $f_2$ are identified,  
$\dist_{H'}(e_1,f_2)\ge 6>d$. 
\end{proof}

Now we are in position to build non-Ramsey hypergraphs which assert more structure in any 
$\Ktt$-free coloring.

\begin{theorem}\label{thm:BEL_gadget}
Let $k\ge 2$ and $t\ge 4$ be integers. Let $H$ be a $3$-uniform hypergraph with $H\not\rightarrow \left(\Ktt\right)_k$ and
let $c\colon E(H)\to[k]$ be a $k$-coloring 
which avoids monochromatic copies of $\Km{t}$. 
Then, there exists a $3$-uniform hypergraph $\cH$ with the following properties:
\begin{enumerate}
\item $\cH\not\rightarrow \left(\Ktt\right)_k$,
\item $\cH$ contains  $H$ as an induced subhypergraph, and  
\item\label{BEL:pattern} for every coloring $\varphi\colon E(\cH)\to [k]$  
without a monochromatic copy of $\Km{t}$, 
the coloring of $H$ under $\varphi$ agrees with the coloring $c$, up to a permutation of the $k$ colors.
\item If there are two vertices $a,b\in V(H)$ with $\deg_H(a,b)=0$ then $\deg_\cH(a,b)=0$ as well.
\item\label{cond:zerocodegree}  If $|V(H)|\ge 4$ then for every vertex $x\in V(\cH)\setminus V(H)$ 
there exists a vertex $y\in V(H)$  such that $\deg_{\cH}(x,y)=0$.
\end{enumerate}
\end{theorem}

\begin{proof}
Let a hypergraph $H$ and a $\Km{t}$-free 
coloring $c$ be given according to the theorem. We take a hypergraph $\cH'$ as asserted to us by Lemma~\ref{lem:rainbow}, 
along with the edges $e_1'$,\ldots, $e'_k$, such that
$V(H)\cap V(\cH')=\emptyset$.
Moreover, let $H'$ be given according to Lemma~\ref{lem:BEL_far_edges},
along with edges $e'$ and $f'$ of distance at least 7.
Then, for every edge $g\in E(H)$, we take a copy $H_g$ of the hypergraph 
$H'$ on a set of new vertices, along with edges $e_g$ and  $f_g$ representing $e'$ and $f'$. 
We identify the edge $g$ with $e_g$ and if $g$ is colored $i$ 
 under the coloring $c$ then we identify 
$f_g$ with $e_i'$.  
 We denote the obtained hypergraph by $\cH$. 

We verify the desired properties one by one.
\begin{enumerate}
 \item 
It is easily seen that every copy $F$ of $\Km{t}$ is contained either in $H$ or in $\cH'$ or in some $H_g$ with $g\in E(H)$.
Indeed, if such a copy contains a vertex $x\in V(H_g)\setminus (e_g\cup f_g)$ for some $g\in E(H)$,
then every other vertex $v\in V(F)$ needs to share an edge with $x$, which by construction needs to be part of $H_g$.
Thus, $V(F)\subseteq V(H_g)$ and $F\subseteq \cH[V(H_g)]=H_g$.
Otherwise, $F$ contains no such vertices $x$, and therefore, $V(F)\subseteq V(H)\cup V(\cH')$.
By construction of $\cH$ we know that $\dist_{H_g}(e_g,f_g)\ge 7$ for all $g\in E(H)$ and thus 
$\deg_{\cH[V(H)\cup V(\cH')]}(u,v)=0$ 
for every $u\in V(H)$ and $v\in V(\cH')$,
which yields $F\subseteq H$ or $F\subseteq \cH'$.

Now, we color $E(H)$ according to $c$. 
 As $V(H)\cap V(\cH')=\emptyset$ we can easily extend $c$ to a 
$\Km{t}$-free coloring of $E(H)\cup E(\cH')$ such that $e'_i$ is colored $i$ for each $i\in[k]$. 
Here we use that by Lemma~\ref{lem:rainbow}, the edges  $e_1'$,\ldots, $e'_k$ have different colors
in any $\Km{t}$-free coloring. 
Moreover, observe that for every $g\in E(H)$ we then have that $e_g$ and $f_g$ receive the same color.

Next, we can extend further the above coloring to a $\Km{t}$-free coloring of $E(\cH)$, by Lemma~\ref{lem:BEL_far_edges} and 
since the $H_g$s have only already colored edges from $\{e_1',\ldots, e'_k\}$ in common. 
Thus, $\cH\not\rightarrow \left(\Ktt\right)_k$.

 \item $H$ occurs as an induced subhypergraph in $\cH$ since $\dist_{H_g}(e_g,f_g)\ge 6$ and thus $e_g\cap f_g=\emptyset$ for all $g\in E(H)$.
 \item Given any $\Km{t}$-free coloring $\varphi$ of $\cH$, it holds by Lemma~\ref{lem:rainbow} that 
$e_1'$,\ldots, $e'_k$ are colored differently. Moreover, by Lemma~\ref{lem:BEL_far_edges},
 the edges $f_g$ and $e_g$ are colored the same (for each $g\in e(H)$) in such a way that the $i$th color class of 
$H$ under $c$ obtains the color $\varphi(e_{i}')$ for each $i\in[k]$. 
 \item Suppose that $\deg_H(a,b)=0$ for some two distinct vertices $a$, $b\in V(H)$. By construction, 
 any two of the  
auxiliary hypergraphs (i.e.\ $\cH'$, $H$, $H_g$s) overlap only in one edge (if at all). This way it follows that $\deg_\cH(a,b)=0$.
\item Finally, take some $x\in V(\cH)\setminus V(H)$. If $x\in V(\cH')\setminus \left(\cup_{g\in E(H)} V(H_g)\right)$, then $\deg_{\cH}(x,y)=0$ for all $y\in V(H)$. If $x\in  V(H_g)$ for some $g\in E(H)$, then again, by construction 
 of $\cH$, we have that $x\not\in g\subset V(H)$ and therefore every $y\in V(H)\setminus g$ satisfies $\deg_{\cH}(x,y)=0$. \qedhere
\end{enumerate}
\end{proof}

\section{\texorpdfstring{Minimum degrees of minimal Ramsey $3$-uniform hypergraphs}{Minimum degrees of minimal Ramsey 3-uniform hypergraphs}}\label{sec:deg_graphs}
Before we prove Theorem~\ref{thm:mindegree}, we first show the 
existence of an appropriate BEL-gadget which will be crucial for the upper bound~\eqref{eq:mindeg} 
in Theorem~\ref{thm:mindegree}.

\begin{lemma}\label{lem:goodgadget}
Let $t\ge 4$ and $k\ge2$ be integers. 
There is a $3$-uniform hypergraph $H$ on $n=k^{10kt^4}$ vertices, which can be 
written as an edge-disjoint union of $k$ $3$-uniform hypergraphs $H_1$, \ldots, $H_k$  with the following properties:
\begin{enumerate}[label=\abc]
\item\label{cond:free} for every $i\in[k]$, $H_i$ contains no copies of $\Km{t}$, and 
\item\label{cond:Ramsey} for any coloring $c$ of the edges of the complete graph $K_n$ with $k$ colors 
   there exists a color $x\in[k]$ and $k$ sets $S_1$, \ldots, $S_k$ that induce 
  copies  of $K_{t-1}$ in color $x$ under the coloring $c$ 
 such that $H_1[S_1]\cong\ldots\cong H_k[S_k]\cong \Km{t-1}$.
\end{enumerate}
\end{lemma}

Before we proceed we state a simple quantitative version of Ramsey's theorem.
\begin{fact}\label{fact:Ramsey_count}
 Let $n\ge r_k(\ell)$. Then, in any $k$-coloring  
of $E(K_n)$ there are at least 
\[
 \frac{n^\ell}{k(r_k(\ell))^\ell}
\]
 monochromatic copies 
of $K_\ell$ in the same color. 
\end{fact}
\begin{proof}
Fix an arbitrary red-blue-coloring $\phi$ of $E(K_n)$. 
 First observe that we find in \emph{any} 
subset of $r_k(\ell)$ vertices of $K_n$ a monochromatic $K_\ell$. We estimate   
pairs of subsets of $[n]$ of the form 
$(R,L)$ with $|R|=r_k(\ell)$, $|L|=\ell$ and $L\subset R$ such that all edges from 
$\binom{L}{2}$ are colored the same. As a lower bound we obtain 
$\binom{n}{r_k(\ell)}$, while the upper bound is the number of monochromatic copies of $K_\ell$ 
under $\varphi$ times the number of $r_k(\ell)$-sets containing a particular copy (which is $\binom{n-\ell}{r_k(\ell)-\ell}$). 
This yields that there are at least 
\[
 \binom{n-\ell}{r_k(\ell)-\ell}^{-1}\binom{n}{r_k(\ell)}=
\frac{n\cdot \ldots\cdot(n-\ell+1)}{r_k(\ell)\cdot\ldots (r_k(\ell)-\ell+1)}\ge \left(\frac{n}{r_k(\ell)}\right)^\ell 
\]
monochromatic $K_{\ell}$s. 
Hence the claim follows.
\end{proof}

The rough idea of the proof of Lemma~\ref{lem:goodgadget} is to take $k$ random hypergraphs 
of appropriate density on the same vertex set and then show that even after deleting 
common edges and edges that lie in copies of $\Km{t}$ we are left with $k$  
edge-disjoint hypergraphs that satisfy condition~\ref{cond:Ramsey}. 
We now turn to the details. 

\begin{proof}[Proof of Lemma~\ref{lem:goodgadget}]
We choose  with foresight 
\begin{equation}\label{eq:probability}
 p:=C\cdot n^{\frac{-6}{(t-1)(t-2)}}, \text{ where }C:=k^{100k/t}\text{ and }n=k^{10kt^4}.
\end{equation}
 
We use the simple upper bound on $r_k(t)\le k^{kt-2k+1}$  
and we define $f(t):=k^{-kt^2}$ so that, with 
Fact~\ref{fact:Ramsey_count}, there are at least $f(t)\cdot n^{t-1}$ monochromatic 
copies of $K_{t-1}$ in one of the  colors in any $k$-coloring of the edges of $K_n$. 

We take $k$ independent random $3$-uniform hypergraphs $H'_1$, \ldots, $H'_k\sim 
H^{(3)}(n,p)$, $i\in[k]$, on the vertex set $[n]$,
and we observe first that
\begin{align*}
& \Exp(e(H_i'\cap H_j'))=\binom{n}{3}p^2, \quad
\Exp(e(H_i'))=\binom{n}{3}p \quad \text{and} \quad\\
& \Exp(\text{number of copies of }K_t^{(3)}\text{ in } H'_i)=\binom{n}{t}p^{\binom{t}{3}}
\end{align*}
for all $i\neq j\in[k]$.

For $i\in[k]$, we denote by $E'_i$ the (random) set of edges 
in $H_i'$ that either belong to some copy of $\Km{t}$ in $H_i'$ or 
to the edge set of some hypergraph $H_j'$, $j\in[k]\setminus\{i\}$. 
We set $H_i:=H'_i\setminus E'_i$. Obviously, $H_1$,\ldots, $H_k$ satisfy~\ref{cond:free}. 
To prove the lemma,
it thus remains to show that~\ref{cond:Ramsey} is satisfied with positive probability. 
This will be immediate from the following two claims. 

\begin{claim}\label{cl:few_bad_edges}
With probability larger than $3/5$, the following holds. 
 Each $H_i'$ contains at most $0.2\cdot f(t)\cdot n^{t-1}p^{\binom{t-1}{3}}$
copies of $K_{t-1}^{(3)}$ that contain an edge from $E'_i$.
\end{claim}

\begin{proof}
Fix an $i\in[k]$. We first consider the number $X$ of copies of $K_{t-1}^{(3)}$
in $H_i'$ that contain an edge $e$ which is part of some copy of $K_{t}^{(3)}$ 
in $H_i'$. For a pair $(T_1,T_2)$ of subsets of $[n]$ with $|T_1|=t-1$ and $|T_2|=t$ we define the indicator 
variable $I_{(T_1,T_2)}$ by
\begin{align*}
I_{(T_1,T_2)}:=\begin{cases}
1,\ \text{ if $H_i'[T_1]\cong\Km{t-1}$ and $H_i'[T_2]\cong\Km{t}$}\\
0,\ \text{ else}
\end{cases}
\end{align*}
and observe that 
\begin{equation}\label{eq:est_X}
X\leq \sum_{s=3}^{t-1} \sum_{\substack{(T_1,T_2):\\ |T_1\cap T_2|=s}} I_{(T_1,T_2)}. 
\end{equation}

By the linearity of expectation it follows that
\begin{multline}\label{eq:estimate_X}
\Exp(X)  \leq \sum_{s=3}^{t-1} n^{t-1} \cdot \binom{t-1}{s}\cdot n^{t-s}\cdot p^{\binom{t-1}{3}+\binom{t}{3}-\binom{s}{3}}\\ 
 \leq 2^tn^{2t-1}p^{\binom{t-1}{3}+\binom{t}{3}}\sum_{s=3}^{t-1} n^{-s}p^{-\binom{s}{3}}.
\end{multline}
Each term  above is dominated by the sum of its first and last summand.
Indeed, let $g(s):=n^{-s}p^{-\binom{s}{3}}$, then for $3\leq s\leq t-2$, we have
\begin{align*}
\frac{g(3)}{g(s)} & = n^{s-3}\cdot p^{\binom{s}{3}-1} = \Big[np^{\frac{s^2+2}{6}}\Big]^{s-3}
 \geq \Big[np^{\frac{s(s+1)}{6}}\Big]^{s-3} \geq \Big[np^{\frac{(t-1)(t-2)}{6}}\Big]^{s-3} \geq 1.
\end{align*}
Thus, we obtain $\Exp(X) \le 2^tn^{2t-1}p^{\binom{t-1}{3}+\binom{t}{3}}\cdot t\cdot \Big(g(3)+g(t-1) \Big)$. 
And we further upper bound $\Exp (X)$ with~\eqref{eq:probability} by
\begin{multline}
 \Exp(X)\le t2^t n^{t-1}p^{\binom{t-1}{3}}\left( n^t p^{\binom{t}{3}}n^{-3}p^{-1}+n^t p^{\binom{t}{3}}n^{-t+1}p^{-\binom{t-1}{3}}\right)\\
\overset{\eqref{eq:probability}}{=} t2^t n^{t-1}p^{\binom{t-1}{3}}\left(C^{\binom{t}{3}}n^{-3}p^{-1}+n^{-2} C^{\binom{t-1}{2}}\right)\\
\overset{\eqref{eq:probability}}{\le} t 2^t  n^{t-1}p^{\binom{t-1}{3}}\left(k^{50kt^2/3}+k^{50kt}\right)n^{-2}\\
\overset{\eqref{eq:probability}}{\le}  2^{t+\log_2 t+1} k^{50kt^2/3} k^{-20kt^4}  n^{t-1}p^{\binom{t-1}{3}}\le \frac{1}{50k} f(t)n^{t-1}p^{\binom{t-1}{3}}.
\end{multline}

So, by Markov's inequality, with probability at least $1-\frac{1}{5k}$ we have,
\begin{equation*}
X \le 0.1 f(t) n^{t-1}p^{\binom{t-1}{3}}.
\end{equation*}

Next, consider the number $Y$ of copies of  $\Km{t-1}$s
in $H_i'$ that contain an edge $e$ from the intersection $E(H_i')\cap E(H_j')$ for a fixed $j\neq i$.
For a subset $S\in \binom{[n]}{t-1}$ and an edge $e\in\binom{S}{3}$
let 
\begin{align*}
I_{(S,e)}:=\begin{cases}
1,\ \text{ if $H_i'[S]\cong K_{t-1}^{(3)}$ and $e\in E(H_{j}')$}\\
0,\ \text{ else}
\end{cases}
\end{align*}
so that $Y \leq \sum_{(S,e)}I_{(S,e)}$.  Then, 
\begin{multline*}
\Exp(Y)\le n^{t-1}\binom{t-1}{3}\cdot p^{\binom{t-1}{3}+1}\overset{\eqref{eq:probability}}{=} n^{t-1}
 p^{\binom{t-1}{3}} \binom{t-1}{3} k^{100k/t} k^{-\frac{60kt^4}{(t-1)(t-2)}}\\ 
\le n^{t-1}
 p^{\binom{t-1}{3}} t^3 k^{25k} k^{-60kt^2} \le \frac{1}{50k^3}f(t) n^{t-1} p^{\binom{t-1}{3}}.  
\end{multline*}

By Markov's inequality, with probability at least $1-\frac{1}{5k^2}$ we then have 
\[
 Y\le \frac{1}{10k} f(t) n^{t-1} p^{\binom{t-1}{3}}.
\]
In particular, with probability at least $3/5$ it holds for all $i\in[k]$ that $H_i'$
contains at most
$0.2\cdot f(t)\cdot n^{t-1}p^{\binom{t-1}{3}}$
copies of $K_{t-1}^{(3)}$ that contain an edge from $E'_i$. Therefore the claim follows.
\end{proof}

\begin{claim}\label{cl:propercopies}
The following holds with probability at least $2/3$.  
 For every coloring $\psi\colon E(K_n)\to [k]$ there is a color
$x$ such that for every $i\in[k]$,
there are at least $0.5 f(t) n^{t-1}p^{\binom{t-1}{3}}$ 
monochromatic copies $F$ of $K_{t-1}$ in color $x$
with $\binom{V(F)}{3}\subseteq E(H_i')$.
\end{claim}

\begin{proof} Fix an $i\in[k]$. 
Let $\psi\colon E(K_n)\to [k]$ be an arbitrary coloring. 
Then there is a color $x$
such that there are at least $f(t)n^{t-1}$ monochromatic copies of $K_{t-1}$
under coloring $\psi$ which all have the same color $x$ (by Fact~\ref{fact:Ramsey_count}). 
We  fix a family $\cF=\{F_1,\ldots,F_m\}$ of exactly $m=f(t)n^{t-1}$ such copies (say lexicographically smallest ones). 
Now, denote with $X_{{\cF},i}$ the number of such $F_j\in\cF$ with
$\binom{V(F_j)}{3}\subseteq E(H_i')$. For every $F_j\in\cF$ let
\begin{equation*}
X_{F_j,i}=\begin{cases}
1,\ \text{ if } \binom{V(F_j)}{3}\subseteq E(H_i')\\
0,\ \text{ else}
\end{cases}
\end{equation*}
and observe that $X_{{\cF},i}=\sum_{F\in{\cF}} X_{F,i}$.
We define   
$\lambda:=\Exp(X_{{\cF},i})=f(t)n^{t-1}\cdot p^{\binom{t-1}{3}}$. 
Observe that by exploiting the choice of $p$ and $n$ in~\eqref{eq:probability} we obtain 
\begin{equation}\label{eq:lambda}
 \lambda= k^{-kt^2} n^{t-1} C^{\binom{t-1}{3}} n^{-t+3}= k^{-kt^2} k^{50k(t-1)(t-2)(t-3)/(3t)} n^2. 
\end{equation}
 
Let 
\[
 \overline{\Delta}_i:=\sum_{\substack{F,F'\in \cF\\ 
\binom{V(F)}{3}\cap \binom{V(F')}{3}\neq\emptyset}}  
\Exp(X_{F,i}X_{F',i}).
\]
Next we estimate $\overline{\Delta}_i$ as follows (since each $X_{F,i}$ counts a copy of the 
complete $3$-uniform hypergraph on the vertex set $V(F)$, we can classify pairs of these copies according to the number $s$ of common vertices): 
\begin{align*}
\overline{\Delta}_i \le |\cF|\sum_{s=3}^{t-1} \binom{t-1}{s}n^{t-1-s}p^{2\binom{t-1}{3}-\binom{s}{3}} 
\le  f(t)\cdot n^{2t-2}p^{2\binom{t-1}{3}} 2^t \sum_{s=3}^{t-1} n^{-s}p^{-\binom{s}{3}},
\end{align*}
and thus exactly as in the previous claim, Claim~\ref{cl:few_bad_edges}, we estimate the sum by 
$t\left(n^{-3}p^{-1}+n^{-t+1}p^{-\binom{t-1}{3}}\right)$, which leads to the upper bound
\begin{multline}\label{eq:Delta}
 \overline{\Delta}_i\le t 2^t \lambda \left(n^{t-1} p^{\binom{t-1}{3}}n^{-3}p^{-1}+n^{t-1} p^{\binom{t-1}{3}} n^{-t+1}p^{-\binom{t-1}{3}}\right)=\\
 t 2^t \lambda \left(C^{\binom{t-1}{3}}(pn)^{-1}+1\right) 
\overset{\eqref{eq:probability}}{=}2^{t+\log_2 t} \lambda \left(k^{\frac{100k}{t}\left[\binom{t-1}{3}-1\right]}k^{-10kt^4+\frac{60 kt^4}{(t-1)(t-2)}}+1\right)\le 2^{2t} \lambda.
\end{multline}

Now with Janson's inequality (see e.g.\ Theorem~2.14 in~\cite{JLR}) we obtain
\begin{multline*}
\PP(X_{{\cF},i}\leq 0.5\lambda) \leq \exp(-\lambda^2/(8\overline{\Delta}_i))
\overset{\eqref{eq:Delta}}{\le} \exp(-2^{-2t-3}\lambda)\\
\overset{\eqref{eq:lambda}}{\le} \exp(-2^{-2t-3} k^{-kt^2+50k(t-1)(t-2)(t-3)/(3t)} n^2)\le \\
\exp(-2^{-2t-3} k^{-kt^2+50kt^2/32} n^2)\le \exp(-k^{-2t-3+9t^2/8} n^2)\le \exp(-k^{7} n^2).
\end{multline*}
This tells us that for the color $x$ 
with probability at least $1-k\exp(-k^7 n^2)$
all graphs $H_i'$, $i\in[k]$, contain at least $0.5\cdot f(t)\cdot n^{t-1}p^{\binom{t-1}{3}}$ 
copies $F$ of $K_{t-1}$ in color $x$ and with  $\binom{V(F)}{3}\subseteq E(H_i')$. 
Since there are $k^{\binom{n}{2}}$ different colorings of $E(K_n)$, we may 
apply the union bound to see that 
the probability that there is a coloring $\psi\colon E(K_n)\to \{\red,\blue\}$ 
not satisfying  the claim
is at most $k^{\binom{n}{2}}\cdot k\exp(-k^7 n^2)<1/3$.
\end{proof}

With positive probability the Claims~\ref{cl:few_bad_edges} and~\ref{cl:propercopies} hold. 
So fix $H_1'$, \ldots, $H_k'$ that satisfy the assertions of these claims. 
Recall that $H_i=H'_i\setminus E'_i$ and we only need to verify~\ref{cond:Ramsey} as  $H_1$,\ldots, $H_k$
 obviously  satisfy~\ref{cond:free}. 
 Let $\psi\colon E(K_n)\to [k]$ be an arbitrary coloring. Claim~\ref{cl:propercopies} 
asserts that there is a color $x$ 
such that for every $i\in[k]$,
there are at least $0.5\cdot f(t)\cdot n^{t-1}p^{\binom{t-1}{3}}$ monochromatic copies  $F$ of  
$K_{t-1}$  in color $x$ and such that 
$\binom{V(F)}{3}\subseteq E(H_i')$. 
By Claim~\ref{cl:few_bad_edges}, for each $i\in[k]$, at most $0.2\cdot f(t)\cdot n^{t-1}p^{\binom{t-1}{3}}$ 
of these copies satisfy $\binom{V(F)}{3}\not\subseteq E(H_i)$, and thus condition~\ref{cond:Ramsey} 
is satisfied.
\end{proof}

\subsection{Proof of Theorem~\ref{thm:mindegree}}
\noindent\emph{A lower bound on $s_{k,1}(\Km{t})$.}
The proof of the  lower bound  is easy. In fact, it follows from the  bound on  
the Ramsey number $r_k(K_t)\ge k^{(1+o(1))t/2}$ and is as follows. 
Take a minimal $k$-Ramsey hypergraph $\cH$ for $\Km{t}$ such that 
$\delta(\cH)=s_{k,1}(\Km{t})$ and let $v\in V(\cH)$ be a vertex of minimum degree. 
By minimality of $\cH$, we have $\cH\setminus\{v\}\not\longrightarrow (\Km{t})_k$ and fix an 
edge coloring $\phi$ that certifies this. Since $\cH\longrightarrow (\Km{t})_k$ it follows that 
the link graph $\link_\cH(v)$ is Ramsey: $\link_\cH(v)\longrightarrow (K_{t-1})_k$. Therefore: 
$s_{k,1}(\Km{t})=\deg(v)\ge \hat{r}_k(K_{t-1})=\binom{r_k(K_{t-1})}{2}\ge k^{(1+o(1))t}$, where 
$\hat{r}_k(K_{\ell})$ is the \emph{size-Ramsey number} for $K_\ell$ and 
it was shown by Erd\H{o}s, Faudree, Rousseau and Schelp~\cite{EFRS78} that $\hat{r}_k(K_\ell)=\binom{r_k(K_\ell)}{2}$.

\noindent\emph{An upper bound on $s_{k,1}(\Km{t})$.}
Let $H$ be the $3$-uniform hypergraph as asserted by Lemma~\ref{lem:goodgadget} along 
with the hypergraphs $H_1$, \ldots,  $H_k$ that satisfy the conditions~\ref{cond:free} and~\ref{cond:Ramsey}. 
We fix the following $\Km{t}$-free $k$-coloring $c$ of $E(H)$: we color all edges from $H_i$ with color $i\in[k]$.  
Let further $\cH'$ be the hypergraph as 
guaranteed by Theorem~\ref{thm:BEL_gadget} for given $H$ and $c$. We define 
the hypergraph $\cH$ by adding to $\cH'$ a new vertex $v$ whose link is 
$\link_\cH(v):=\binom{V(H)}{2}$.  So $\deg_\cH(v)=\binom{n}{2}< k^{20kt^4}$
 as asserted by Lemma~\ref{lem:goodgadget}. 
In the following we argue that
$\cH'\not\longrightarrow (\Km{t})_k$ but $\cH\longrightarrow (\Km{t})_k$. 
It then follows immediately that every Ramsey subhypergraph of $\cH$ (in particular minimal Ramsey subhypergraph of $\cH$) 
for $\Km{t}$ needs to contain the vertex $v$, whose degree is less than $k^{20kt^4}$. 
Thus, once these two properties are proven, the upper bound follows.

In fact, $\cH'\not\longrightarrow (K_t^{(3)})_k$ is asserted 
by Theorem~\ref{thm:BEL_gadget}. So, we only need to focus on showing that
$\cH\longrightarrow (K_t^{(3)})_k$. For contradiction, suppose that 
there is a coloring $\phi\colon E(\cH)\to [k]$ without
monochromatic copies of $K_t^{(3)}$. We then know by the Property~\eqref{BEL:pattern} 
from Theorem~\ref{thm:BEL_gadget} that $E(H_1)$, \ldots,  $E(H_k)$ 
are all colored monochromatically, but in different colors. 
W.l.o.g.\ we may assume that, for each $i\in[k]$, $H_i$ is colored with the color $i$.  
Now, we define a coloring $\psi\colon \binom{V(H)}{2}\to [k]$ 
with $\psi(\{u_1,u_2\})=\phi(\{u_1,u_2,v\})$. 
Then, according to Lemma~\ref{lem:goodgadget} there is a color $x$ 
and the sets $S_1$, \ldots, $S_k\in\binom{V(H)}{t-1}$ such that 
$\binom{S_1}{2}$, \ldots, $\binom{S_k}{k}$ are monochromatic under $\psi$ in color $x$, 
 while for every $i\in[k]$ we have that  $H[S_i]\cong \Km{t-1}$ is colored $i$.  
 But this implies immediately that
 we found a monochromatic clique
 $\cH[S_x\cup\{v\}]\cong\Km{t}$ in color $x$. A contradiction.
\hfill $\Box$

\section{\texorpdfstring{Minimum codegrees of minimal Ramsey $3$-uniform hypergraphs}{Minimum codegrees of minimal Ramsey 3-uniform hypergraphs}}\label{sec:codeg_graphs}
 In this section we prove Theorem~\ref{thm:main_result_codeg} by showing that 
 $s_{2,2}(K_t^{(3)})=0$ and that  $s'_{2,2}(K_t^{(3)})=(t-2)^2$. 
Our proof strategy is similar to that of~\cite{BEL76,FGLPS14}: for the lower bound we rather provide an adhoc argument, while 
for the upper bound we employ the BEL-gadgets, Theorem~\ref{thm:BEL_gadget}, combined 
with a natural construction that we ``plant'' via a BEL-gadget (which is an almost Ramsey hypergraph).

\begin{proof}[Proof of Theorem~\ref{thm:main_result_codeg}]\mbox{}\\
\noindent\emph{Lower bound argument for $s'_{2,2}$.}
We first prove that $s'_{2,2}(\Km{t})\geq (t-2)^2$. Take a minimal 
$2$-Ramsey hypergraph $H$ for $\Km{t}$. Fix any two vertices $u$ and $v\in V(H)$ with $\deg_H(u,v)>0$. 
 We aim to show that $\deg_H(u,v)\geq (t-2)^2$. So, assume the opposite, i.e.\ $\deg_H(u,v)<(t-2)^2$. 

Let $H'$ be the subhypergraph obtained
from $H$ by deleting all edges containing both vertices $u$ and $v$. Since $H$ is Ramsey-minimal, 
$H'\not\rightarrow \left(\Ktt\right)_2$. Thus, there is a coloring $c$ with red and blue
of $E(H')$ which does not create a monochromatic copy of $\Km{t}$.
Define $N(u,v):=\{w\in V(H):\ \{u,v,w\}\in E(H)\}$, thus  $\deg_H(u,v)=|N(u,v)|$.
Take a longest sequence $B_1$,\ldots,$B_k$  of vertex disjoint sets of size $t-2$ in $N(u,v)$, 
such that both $B_i\cup \{u\}$ and $B_i\cup \{v\}$ span only blue edges under the coloring $c$ in $H$.
By assumption on the codegree $\deg_H(u,v)$, we know that $k<t-2.$ 

Next we can extend the coloring $c$
as follows. For each edge $e=\{u,v,w\}\in E(H)$ with $w\in\bigcup B_i$ we set $c(e)=\red$, while 
for all other edges $e=\{u,v,w\}\in E(H)$ we set $c(e)=\blue$. We claim that under this coloring
 there is no monochromatic copy of $K_t^{(3)}$ in $H$. Indeed, if there were a monochromatic
subgraph $F$ isomorphic to $\Km{t}$, then necessarily $u,v\in V(F)$ (since $E(H')$ were colored without 
monochromatic $\Km{t}$). If $F$ is red, then by construction $F$ can have at most one
vertex from each of the sets $B_i$ and no vertex from $N(u,v)\setminus \bigcup B_i$,
so $|V(F)|<t$, a contradiction. If $F$ is blue, then it cannot contain vertices from $\bigcup B_i$,
and therefore $V(F)\subseteq (N(u,v)\setminus \bigcup B_i)\cup \{u,v\}$. 
But then, we could extend the sequence of $B_i$s by the set $V(F)\setminus \{u,v\}$,
in contradiction to its maximality.
So, under the assumption $\deg_H(u,v)<(t-2)^2$ we conclude that 
$H\not\rightarrow (K_t^{(3)})_2$, a contradiction. Thus, we need to have
$\deg_H(u,v)\geq (t-2)^2$ for every $u,v\in V$ with $\deg_H(u,v)>0$. Therefore, $s'_{2,2}(\Km{t})\ge (t-2)^2$.

\noindent\emph{Upper bound argument for $s'_{2,2}$.}
First we provide a hypergraph $H$ with a prescribed coloring of
$E(H)$ without a monochromatic $\Km{t}$. 
We set $V(H):=[(t-2)^2]\cup\{a,b\}$ and we further  partition the vertices of
 $[(t-2)^2]$  into $(t-2)$ equal-sized sets $V_1$,\ldots, $V_{t-2}$. Next we 
 choose the edges for 
$H$ as follows:
\begin{equation}\label{eq:edges_H}
 \begin{split}
E(H):=&\bigcup_{i}^{t-2}\binom{V_i}{3}\cup\left\{e\cup\{w\}\colon e\in\binom{V_i}{2}\text{ for some }i\in[t-2], w\in\{a,b\}\right\} \\ 
&\cup\left\{f\colon f\in\binom{[(t-2)^2]}{3}, |f\cap V_i|\le 1\,\, \forall i\in[t-2]\right\}\\ 
&\cup
\left\{e\cup\{w\}\colon e\in\binom{[(t-2)^2]}{2}, |e\cap V_i|\le 1\,\, \forall i\in[t-2], w\in\{a,b\} \right\}.   
 \end{split}
\end{equation}
Thus, $H$ is obtained from the clique $\Km{(t-2)^2+2}$ on the vertex set $\bigcup V_i\cup\{a,b\}$, where 
we delete all edges that contain both $a$ and $b$ and moreover we delete all edges that cross exactly two different $V_i$s
and contain neither $a$ nor $b$.
Next we provide  a red-blue-coloring $c$ of  the edges of $H$ as follows: 
the edges  contained in $V_i\cup\{a\}$ and in $V_i\cup\{b\}$ for $i\in[t-2]$ are 
colored \emph{blue}, while the other edges of $H$ are colored \emph{red} -- thus the edges in the first line 
of~\eqref{eq:edges_H} are colored blue, while the edges defined in the second and third line of~\eqref{eq:edges_H} 
are colored red. It is immediate that such a coloring does not yield 
a monochromatic copy of $\Km{t}$. Indeed, a blue copy of $\Km{s}$ cannot use vertices from different
sets $V_i$ and, since $\deg_H(a,b)=0$, it also cannot contain both vertices $a,b$, which gives $s\leq t-1$.
Similarly, a red copy of $\Km{s}$ can use at most one vertex from each $V_i$ and, as $\deg_H(a,b)=0$, it also cannot contain
both vertices $a,b$, which again gives $s\leq t-1$.

Applying Theorem~\ref{thm:BEL_gadget} to the colored hypergraph $H$ for this coloring $c$, we obtain a $3$-uniform hypergraph $\cH$ which contains 
$H$ as an induced hypergraph, which is not 
$2$-Ramsey for $\Km{t}$ and such that  any red-blue $\Km{t}$-free coloring $\phi$ of $E(\cH)$ agrees on $E(H)$
with the coloring $c$ up to permutation of the two colors. Also, Theorem~\ref{thm:BEL_gadget} asserts that $\deg_\cH(a,b)=0$. 
Next we define $\cH'$ by adding to $\cH$ all $(t-2)^2$ edges $\{a,b,u\}$ where $u\in[(t-2)^2]$. 

Let us see why $\cH'\longrightarrow (\Km{t})_2$. Fix any coloring $\phi$ of $E(\cH')$ and assume that no copy of 
$\Km{t}$ is monochromatic in $\cH'$ under $\phi$. Since $\cH\subset \cH'$, it follows that the color pattern  $c$ 
as described above (up to permutation) is enforced in $H$. Assume w.l.o.g.\ that 
$E(H)$ is colored according to $c$. Then if there is a set $V_i$ such that all 
edges $\{v,a,b\}$ are colored blue for all $v\in V_i$ this would yield a blue copy of $\Km{t}$. 
So, assume that for every $V_i$ there is at least one edge  $\{v_i,a,b\}$ which is colored red for some $v_i\in V_i$. 
Then $\{a,b,v_1,\ldots,v_{t-2}\}$ forms a red clique $\Km{t}$. Thus, in any case, we find a monochromatic copy
of $\Km{t}$, i.e. $\cH\longrightarrow (\Km{t})_2$. 
Moreover, since $\cH$ is not $2$-Ramsey for $\Km{t}$, any minimal $2$-Ramsey subhypergraph of $\cH'$  
must contain edges that contain both $a$ and $b$. This shows $s'_{2,2}(\Km{t})\leq (t-2)^2$.

In fact, notice that by the previous discussion of the lower bound on $s'_{2,2}$,
any such minimal $2$-Ramsey subhypergraph of $\cH'$  
must contain all the $(t-2)^2$ edges that contain both $a$ and $b$. This will be important in the following proof.

\noindent\emph{Showing $s_{2,2}(\Km{t})=0$.} This looks surprising at the first sight since taking $\Km{n}$ with 
$n=r_2(\Km{t})$ and then deleting all edges that contain two distinguished vertices gives a non-Ramsey 
hypergraph (which suggests  $s_{2,2}(\Km{t})>0$). However this is not the case and it 
will follow from the above construction of the hypergraph $\cH'$.

 As argued above, \emph{any} 
minimal Ramsey subhypergraph  of $\cH'$ for $\Km{t}$ has to contain \emph{all} 
$(t-2)^2$ edges that contain $a$ and $b$. Thus, any such minimal hypergraph $\cH''$ contains
all vertices of $H$. Next we argue that $\cH''[V(H)]\not\longrightarrow (\Km{t})_2$. Indeed, by construction 
of $\cH'$, we observe that $\cH'[V(H)]\supseteq \cH''[V(H)]$ contains exactly 
$(t-2)+(t-2)^{t-2}$ copies of $\Km{t}$, namely exactly $(t-2)$ ones 
that are induced on $V_i\cup\{a,b\}$ for some $i\in [t-2]$, and $(t-2)^{t-2}$ ones that contain one vertex from each of the $V_i$s and additionally $a$ and $b$. There are no further copies of $\Km{t}$ since $H[\bigcup V_i]$ contains only copies of $\Km{t-2}$ 
which either cross all $V_i$s or are equal to some $H[V_i]$. It is now easy to see that $\cH'[V(H)]\not\longrightarrow (\Km{t})_2$ as follows. We can color the edges of $\cH''[V(H)]$ uniformly at random with colors red and blue. Then, the expected number of 
monochromatic copies of $\Km{t}$ is $[(t-2)+(t-2)^{t-2}]\cdot 2^{1-\binom{t}{3}}<1$, as $t\geq 4$,
i.e. there exists a 2-coloring which avoids monochromatic copies of $\Km{t}$.

Thus, $\cH''$ has to contain at least one further vertex $x\not\in V(H)$. Then, since $|V(H)|=(t-2)^2+2\ge 6$, it follows 
by Property~\eqref{cond:zerocodegree} of Theorem~\ref{thm:BEL_gadget} that 
there exists a vertex $y\in V(H)$ such that $0=\deg_{\cH'}(x,y)\ge \deg_{\cH''}(x,y)$. Therefore, $s_{2,2}(\Km{t})=0$.
\end{proof}

\section{concluding remarks}
In this paper we studied the smallest minimum degree and codegree of minimal Ramsey $3$-uniform 
hypergraphs for complete hypergraphs $\Ktt$, $t\ge 4$. In particular we showed that the smallest  minimum degree 
 $s_{2,1}(\Ktt)$ of minimal $2$-Ramsey $3$-uniform hypergraph lies between $2^t$ and $2^{40 t^4}$. It would be interesting to determine 
the right order of the exponent. We leave the study of minimal Ramsey $r$-uniform hypergraphs for $r\ge 4$  for future work.

\bibliographystyle{amsplain}
\bibliography{references}

\end{document}